\documentclass[11pt]{article}

\usepackage{amssymb,amsfonts,amsmath,amsopn,amstext,amscd,latexsym,xy,color,graphicx, verbatim, dsfont, enumitem, tikz, tikz-cd, xfrac, amsthm, framed, hyperref,xspace, thm-restate, authblk}

\usepackage[margin=1in]{geometry}

\newtheorem{thm}{Theorem}
\newtheorem{lem}[thm]{Lemma}
\newtheorem{prop}[thm]{Proposition}

\newtheorem{conj}[thm]{Conjecture}
\newtheorem{prob}[thm]{Problem}

\theoremstyle{definition}

\newtheorem{defn-lem}[thm]{Definition/Lemma}

\theoremstyle{remark}

\newcommand{\cP}{\mathcal{P}}

\newcommand{\ignore}[1]{}

\newcommand{\dd}{\ensuremath d_{1}}

\newcommand{\gnpd}{\ensuremath G_{n,p,d}}

\newcommand{\er}{Erd\H{o}s-R\'{e}nyi\xspace}

\newcommand{\pr}{\ensuremath \mathbb{P}}
\newcommand{\PP}{\mathbb{P}}
\newcommand{\E}{\mathbb{E}}

\newcommand{\reminder}[1]{}  

\newcommand{\rel}[2]{\ensuremath #1 \rightsquigarrow #2}

\title{\vspace*{-.7in}Evolution of locally dependent random graphs}
\author{Joshua Brody$^1$, Pat Devlin$^1$, Aditi Dudeja$^2$, Emmi Rivkin$^1$}

\date{%
  $^1$ Swarthmore College, Swarthmore, PA, United States\\%
  $^2$ University of Salzburg, Salzburg, Austria\\ \ \\%
  May 15, 2024
}
\begin{document}
\maketitle

\renewcommand{\thefootnote}{\fnsymbol{footnote}}
\setcounter{footnote}{0}
\renewcommand{\thefootnote}{\arabic{footnote}}

\begin{abstract}
In this paper we study $d$-dependent random graphs---introduced by Brody and Sanchez~\cite{mario_paper}---which are the family of random graph distributions where each edge is present with probability $p$, and each edge is independent of all but at most $d$ other edges.  For this random graph model, we analyze degree sequences, jumbledness, connectivity, and subgraph containment.  Our results mirror those of the classical Erd\H{o}s--R\'enyi random graph, which are recovered by specializing our problem to $d=0$, although we show that in many regards our setting is appreciably more nuanced.  We survey what is known for this model and conclude with a variety of open questions.
\end{abstract}

\section{Introduction}
Since the seminal papers of Erd\H{o}s and R\'enyi \cite{erdds1959random, erdHos1960evolution}, random graphs have been among the most researched topics in discrete mathematics and theoretical computer science.  Random graphs have had a wide range of applications including computer science, discrete math, statistical mechanics, and modelling.  By far the most well-studied random graph distribution is the \textit{Erd\H{o}s--R\'enyi model} which is a random graph on $n$ vertices where each edge is present independently with probability $p$.  After decades of active research, this distribution is now quite well understood---see texts \cite{alon-spencer, bollobas_book,jlr-book}---and in fact recent tools for understanding this model have become so sophisticated as to render many low-level questions almost trivial \cite{expectationThresholds, jinyoung-expectation}.

Although the Erd\H{o}s--R\'enyi random graph is the most popular random graph distribution in the literature, it is unable to model a number of important settings where the edges depend on each other, and a wide variety of other random graph models have been proposed over the years, see e.g., the textbook \cite{frieze} or the survey \cite{super-survey}.  Of particular relevance for us is \cite{alon-Nussboim}, where Alon and Nussboim considered $k$-wise independent random graphs.  In this model, each edge is present with probability $p$, but the independence assumption is loosened to say merely that any subset of $k$ edges is independent (with no assumptions about larger sets).  In that setting, Alon and Nussboim were able to establish a number of results including a careful estimation of eigenvalues.  However, because the family of $k$-wise independent graph distributions is relatively poorly behaved, their results were necessarily weaker than those for the Erd\H{o}s--R\'enyi graph.

Motivated by an application to number-on-the-forehead communication complexity, Brody and Sanchez \cite{mario_paper} introduced the \textit{locally dependent} random graph model, which is the focus of the current paper.  In this model, a $d$-dependent random graph $\gnpd$ is any random graph distribution on $n$ vertices where each edge appears with probability $p$, and each edge depends on at most $d$ others.  

Said more precisely, a \emph{dependency graph} for a set of random variables $X_1,\ldots, X_n$ is a graph $H=(V,E)$ where vertices correspond to the random variables, and edges correspond to dependencies:  each $X_i$ is independent of all $X_j$ except those for which $(i,j)\in E$.  Variables $\{X_i\}$ are $d$-\emph{locally dependent} if there exists a dependency graph for $\{X_i\}$ with maximum degree $d$.  
%
This notion of \textit{local dependence} coincides with that considered in both the well-known Lov\'asz local lemma \cite{lll,moser-tardos} as well as Janson's inequality \cite{janson}.  A 
%
$d$-dependent random graph $\gnpd$ is a random graph distribution on $n$ vertices, where each edge is present with probability $p$ and the set of edges are $d$-locally dependent.

The family of $d$-dependent random graphs is a relatively new but natural model, which in some sense interpolates between the Erd\H{o}s--R\'enyi random graph ($d=0$) and every other distribution of random graphs ($d = \genfrac(){0pt}{2}{n}{2}$), provided each edge appears with probability $p$.

One of the particularly interesting aspects of this model is that because it encompasses an entire family of distributions, it does not typically exhibit thresholds for monotone properties in the simple way that the Erd\H{o}s--R\'enyi model does.  For example, if $G$ is an Erd\H{o}s--R\'enyi random graph, then viewing $p$ as a function of $n$, we have that $\PP(G  \text{ is connected})$ tends to 0 if $p \ll \ln(n)/n$, and the probability tends to 1 if $p \gg \ln(n)/n$ (see \cite{bollobas_book} for this and much more precise results).

In the setting of $d$-dependent random graphs, it's unreasonable to expect such a simple result would be true since we are considering an entire \textit{family of distributions} simultaneously.  That is to say, in our setting a typical \textit{threshold} result ought to have three regimes separated by distinct lower and upper thresholds, $L(n,d) \leq U(n,d).$  Namely, (i) if $p \ll L(n,d)$, then all $d$-dependent distributions are very unlikely to have the desired property; (ii) if $p \gg U(n,d)$, then all distributions are very likely to have the desired property; and (iii) for intermediate values of $p$, there are some $d$-dependent distributions that are likely to have the property and others that are not.  The gap---if any---between (i) and (ii) would necessarily be some function of $n$ and $d$ which (by a classical result of \cite{thresholds}) essentially vanishes when $d=0$.  One of the main results of this paper addresses these lower and upper thresholds for connectivity.  We discuss this further in our conclusion section.

As an example of a related phenomenon, in \cite{mario_paper} Brody and Sanchez initiated the study of $d$-dependent random graphs, where their main result is the following:
\begin{prop}[Theorems 1 and 3 of \cite{mario_paper}]\label{prop:clique number} Say $d/\sqrt{n} \ll p \leq 1/4$.  There are positive constants $C_1, C_2$ such that for any $d$-dependent random graph $G \sim \gnpd$, with probability tending to 1
\[
C_1 \dfrac{\ln(n)}{\ln(1/p)} \leq \text{(clique number of $G$)} \leq C_2 \dfrac{(d+1)\ln(n)}{\ln(1/p)}.
\]
\end{prop}
The upper bound follows simply by noting that the expected number of cliques on $k$ vertices is at most ${n \choose k} p^{{k \choose 2}/(d+1)}$, since any collection of $m$ edges must contain a subset of size at least $m/(d+1)$ whose presence are independent events.  Their lower bound was much more involved, and it amounted to estimating the mean and variance of a related random variable and applying Azuma's inequality.  Interestingly, this shows that when $d$ is not too large (say $d \ll p \sqrt{n}$), the clique number of a $d$-dependent random graph cannot be fundamentally lower than what it would be in the Erd\H{o}s--R\'enyi case.  They also exhibit graphs with clique number roughly $C \sqrt{d+1} \ln(n)/\ln(1/p)$.  It is unknown if the upper bound in Proposition \ref{prop:clique number} can be attained.

The authors of \cite{mario_paper} used the above to prove the immediate corresponding inequalities regarding the chromatic number of $\gnpd$, which in turn gave improved bounds for the multiparty pointer jumping problem, a canonical problem in number-on-the-forehead communication complexity.  
%
They also noted that very little can be said about the clique number of dependent random graphs when $d = \Omega(n)$.
For instance, if $p = 1/2$, then the random bipartite graph (where each $x$ and $y$ are in the same partition with probability exactly $1/2$) would be a $2n$-dependent random graph which is always bipartite.  And its complement would be a $2n$-dependent random graph which is the disjoint union of two cliques (at least one of which must have at least $n/2$ vertices).  Thus, the lower bound of Proposition \ref{prop:clique number} cannot be true for \textit{all} values of $d$.

%
%
\subsection*{Our results}
We now turn our attention to original results on $d$-dependent random graphs.  As this model is still relatively new, we focus on several of the most fundamental questions in random graph theory: degree sequence, jumbledness, connectivity, and subgraph containment.

Throughout, we say that an event happens \textit{almost surely} to mean that it happens with probability tending to 1 as the relevant parameter (typically $n$) tends to infinity.  We use $\ln$ to denote the natural (base $e$) logarithm.

Our first result is the fact that for many choices of parameters, almost surely every vertex has degree very nearly $np$.

\begin{restatable}[Degree concentration]{thm}{degreeConcentrationTheorem}\label{theorem:degree concentration}
Say $G \sim \gnpd$ is any $d$-dependent random graph distribution.  If $p \geq (d+1) \dfrac{\ln(n)}{n}$, then almost surely, every vertex of $G$ has degree between $np - 4\sqrt{np(d+1)\ln(n)}$ and $np + 4\sqrt{np(d+1)\ln(n)}$.
\end{restatable}

In fact, this can be generalized in a certain sense by arguing that with high probability any $d$-dependent random graph exhibits a \textit{jumbledness} property that the number of edges between any two subsets of vertices is roughly what it would be in the Erd\H{o}s--R\'enyi setting.  Jumbledness, first introduced by Thomason \cite{thomason}, is a convenient graph property that in many ways allows such graphs to be treated as Erd\H{o}s--R\'enyi graphs \cite{krivelevich-survey}.

\begin{restatable}[Jumbledness]{thm}{jumblednessTheorem}\label{theorem:jumbledness}
Suppose $C \geq 10$, and $G \sim \gnpd$ is drawn from any $d$-dependent random graph distribution.  Further suppose that almost surely the maximum degree of $G$ is at most $(C^2/14)np$---which would be implied if $p \geq 14/C^2$ or $p \geq (d+1) \ln(n) / n$.  Then almost surely, all subsets of vertices $A,B$ simultaneously satisfy
\[
\Big| e(A,B) - p |A| |B| \Big | \leq C \sqrt{|A||B|np (d+1)},
\]
where $e(A,B)$ counts the number of edges of $G$ having one endpoint in $A$ and the other in $B$---with the usual convention that edges contained in $A \cap B$ are counted twice.
\end{restatable}
In the language of pseudorandom graphs, the above result shows that the $d$-dependent random graph is almost surely $(p, C\sqrt{np(d+1)})$-jumbled.  The next example shows that this is best possible up to the choice of $C$ and that we must have $C \geq 1$.

\begin{restatable}[Jumbledness example]{thm}{jumblednessExample}\label{example: jumbledness}Suppose $np \to \infty$ and $d < 0.99n$.  Then there is a $d$-dependent random graph distribution such that almost surely, there are sets $A$ and $B$ where
\[
e(A,B) - p|A||B| > (1-o(1))(1-p)\sqrt{1-d/n} \sqrt{|A||B|np (d+1)}.
\]
\end{restatable}

We now turn our attention to connectivity.

\begin{restatable}[Connectivity]{thm}{connectivityTheorem}\label{theorem:connectivity}
Let $G \sim \gnpd$ be any $d$-dependent random graph distribution.
\begin{itemize}
    \item[(i)] Suppose $f(n)$ is any function for which $f(n) \to \infty$.  If $p > (d+1)\dfrac{\ln(n) + f(n)}{n}$, then $G$ is almost surely connected.
    \item[(ii)] If $np \to 0$, then almost surely $G$ will not be connected.
\end{itemize}
\end{restatable}
As for the optimality of each of these bounds, we have:

\begin{restatable}[Connectivity examples]{thm}{connectivityExample}\label{theorem:connectivity examples}
There are distributions with the following properties.
\begin{itemize}
    \item[(i)] There is a $d$-dependent random graph distribution with the property that for each fixed $\varepsilon > 0$, if $p < (1-\varepsilon) (d+1) \dfrac{\ln(n/\sqrt{d+1})}{n}$ and $d\ln(n)/n \to 0$, the graph is almost surely not connected.
    \item[(ii)] There is a $d$-dependent random graph distribution (namely, the Erd\H{o}s--R\'enyi graph) such that for any function $f(n) \to \infty$, if $p > \dfrac{\ln(n) +f(n)}{n}$, then the graph is almost surely connected.
\end{itemize}
\end{restatable}
Thus, item (i) of Theorem \ref{theorem:connectivity} is nearly best possible up to a factor of $1/\sqrt{d+1}$ appearing in the logarithm.  In fact, if $d \leq n$, then the bounds given in item (i) of the above theorems differ by at most a multiplicative factor of 2.  As for the lower threshold, there is a gap where it is uncertain whether or not there are any almost surely connected $d$-dependent random graphs with $1/n \leq p \leq \ln(n)/n$.  We feel that closing this gap would be of particular interest.  We discuss this general phenomenon and related open problems in the conclusion.

Finally, we conclude with a generalization of a result of Bollob\'as \cite{bollobas_subgraphs} addressing thresholds for when a $d$-dependent random graph is likely to contain small subgraphs.
\begin{restatable}[Subgraph containment]{thm}{subgraphTheorem}\label{theorem:subgraph counts}
Suppose $H = H_n$ is any sequence of graphs.  For each graph $\Gamma$, let $V(\Gamma)$ denote its vertex set, $E(\Gamma)$ its edge set, and let $f(\Gamma)$ denote \textit{its edge cover number}, i.e., the minimum number of edges in a set $A \subseteq \Gamma$ such that every edge of $\Gamma$ shares an endpoint with some edge of $A$.  Define the quantity
\[
\Phi(H) := \displaystyle \sum_{\emptyset \neq \Gamma \subseteq H} \left( \dfrac{20 |V(H)|}{n}\right)^{|V(\Gamma)|} \dfrac{(d+1)^{f(\Gamma)}}{p^{|E(\Gamma)|}},
\]
where the sum is taken over non-empty subgraphs of $H$.  Finally suppose $d |E(H)| |V(H)|/n < 0.9$.

For any $d$-dependent graph distribution $G \sim \gnpd$, the probability that $G$ contains no copies of $H$ is at most $10 \Phi(H)$.  In particular, define $m(H) = \max_{\emptyset \neq \Gamma \subseteq H} |E(\Gamma)| / |V(\Gamma)|$.  If $H$ is fixed and $p (n/(d+1))^{1/m(H)} \to \infty$, then $G$ almost surely contains a copy of $H$.
\end{restatable}
Note that if $d=0$, this recovers the result of Bollob\'as for the Erd\H{o}s--R\'enyi random graph.  In that setting (when $H$ is fixed) the Erd\H{o}s--R\'enyi random graph almost surely has no copies of $H$ when $p n^{1/m(H)} \to 0$.  As an example in our model, when $H=K_3$ is a triangle, we have
\[
\Phi(K_3) = C \left( \dfrac{(d+1)^1}{n^2 p} + \dfrac{(d+1)^2}{n^3 p^2} + \dfrac{(d+1)^2}{n^3 p^3} \right).
\]
Thus, if $p \gg (d+1)^{2/3} / n$, then any $d$-dependent random graph will almost surely contain a triangle.

\subsection*{Outline of paper}
We begin in Section \ref{sec:notation} with common notation used throughout.  Continuing in Section \ref{section:examples}, we provide several examples of $d$-dependent graph distributions including a proof of Theorems \ref{example: jumbledness} and \ref{theorem:connectivity examples}.  Next, in Section \ref{section:degrees and jumbled} we prove Theorems \ref{theorem:degree concentration} and \ref{theorem:jumbledness}.  Section \ref{section:connectivity proof} is devoted to a proof of Theorem \ref{theorem:connectivity}, and we prove Theorem \ref{theorem:subgraph counts} in Section \ref{section:subgraph counts}.  Finally, we list a variety of open questions and conjectures in the conclusion.  \reminder{For possible ease of reading, we remind the reader of each theorem prior to its proof.}

\section{Common notation and useful facts}\label{sec:notation}
Throughout, we use $\ln(x)$ to denote the natural log of $x$.  We say an event happens almost surely to mean that it happens with probability tending to 1 as the relevant parameter (typically $n$) tends to infinity. Following notation of Janson's inequality \cite{janson}, we let $\dd$ denote $d+1$.

As a common approximation, we use the fact that for all $x$ we have $1+x \leq e^{x}$ and moreover if $-1/2 \leq x$ we have $e^{-2x} \leq e^{x-x^2} \leq 1+x$.  Combining these, we have that if $x \to 0$, then $1+x = e^{x(1+o(1))}$.  Moreover, for any integers $1 \leq k \leq m$, we use the fact $(m/k)^k \leq {m \choose k} \leq (em/k)^k$ (for details on all of these, see, e.g., the first chapter of \cite{bollobas_book}).

Let us now recall the following versions of Janson's inequality \cite{janson} specialized for our situation.
\begin{lem}[Janson]\label{lemma:Janson inequalities}
Let $G$ be any $d$-dependent random graph, and let $A, B$ be any subsets of vertices.  Define $e(A,B)$ to be the number of edges of $G$ between $A$ and $B$, with the usual convention that edges inside $A \cap B$ are counted twice.  Define $\mu = \E[e(A,B)] = |A||B|p$.  For all $t > 0$ we have
\begin{eqnarray}
\PP\left( \Big | e(A,B) - \mu \Big| \geq t \right) &\leq& 2 \exp \left[ \dfrac{-8 t^2}{50 \dd (\mu + t/3) }\right],\label{Janson:Bernstein}\\
\PP \left( \Big | e(A,B) - \mu \Big | \geq t \right) &\leq& 2\exp \left[ \dfrac{-\mu}{2 \dd } \varphi\left(\dfrac{4t}{5\mu} \right) \right],\label{Janson:with phi}
\end{eqnarray}
where $\varphi(x) = (1+x)\ln(1+x)-x$.
\end{lem}
\begin{proof}
These inequalities are a direct application of Theorem 2.3 of \cite{janson} in the following way.  We may treat $e(A,B)$ as a sum of $|A||B|$ indicator random variables (with edges inside $A \cap B$ appearing twice in this sum), where each random variable has mean $p$ and each random variable depends on fewer than $2 \dd$ terms in the summation [multiplying by 2 to account for the fact that some edges may have been duplicated in this sum].
\end{proof}

\section{Examples of $d$-dependent random graphs}\label{section:examples}
Here we provide examples of two different $d$-dependent random graph distributions and use them to prove Theorems \ref{example: jumbledness} and \ref{theorem:connectivity examples}.  In each case, for convenience we assume that all large numbers are integers (since the errors associated with rounding do not change the validity of our results).
\subsection*{Proof of Theorem \ref{example: jumbledness}: jumbledness example}
\reminder{\jumblednessExample*}
\begin{proof}
Partition the vertex set $V$ into sets $V = S \cup T$ where $|S| = \dd$.  For each $x \in T$, we fully correlate all the edges from $x$ to $S$.  That is, $x$ is adjacent to all of $S$ with probability $p$ and otherwise it is adjacent to none of the vertices of $S$.  All other edges of this graph are independent.

The key is that every vertex of $S$ has the exact same set of neighbors in $T$.  Let $B \subseteq T$ denote the vertices of $T$ that are adjacent to at least one (and hence every) vertex of $S$.  Then $|B|$ is a binomial random variable with mean $|T|p$ and variance $|T|p(1-p)$.  Applying Chebyshev's inequality
\begin{eqnarray*}
\PP\Big(|B| \leq |T|p (1-(np)^{-1/4}) \Big) &=& 
\PP\Big(|B| - \E[|B|] < - |T|p (np)^{-1/4} \Big)  \leq \dfrac{\text{Var}(|B|)}{\Big ( |T|p (np)^{-1/4} \Big) ^2}\\
&=& \dfrac{|T|p(1-p)}{(|T|p)^2/\sqrt{np}} \leq \dfrac{\sqrt{np}}{|T|p} = \dfrac{1}{(1-\dd/n) \sqrt{np}} \to 0.
\end{eqnarray*}
Thus, almost surely $|B| \geq np(1-\dd/n)(1-o(1))$, and as a result
\begin{eqnarray*}
e(S,B) - |S||B|p &=& (1-p)|S||B| = (1-p) \sqrt{|S| |B| \cdot (|S||B|)}\\
&\geq& (1-o(1))(1-p)\sqrt{1-d/n} \sqrt{|S||B|np \dd},
\end{eqnarray*}
as we were trying to show.
\end{proof}

\subsection*{Proof of Theorem \ref{theorem:connectivity examples}: non-connected example}
\reminder{\connectivityExample*}

\begin{proof}Claim (ii) is a well-known result in random graphs---which also follows immediately from Theorem \ref{theorem:connectivity}.  Here we prove claim (i) of Theorem \ref{theorem:connectivity examples}.  For this, suppose $\varepsilon>0$ is fixed, $p \leq (1-\varepsilon) \dfrac{\dd \ln(n/\sqrt{\dd})}{n}$, and $\dd \ln(n)/n \to 0$, which in turn implies that $p \to 0$ as well.  We will prove that there is a $d$-dependent distribution which is almost surely not connected.

Partition the vertex set $V$ as $V = A \cup B$ where $|A| = \dfrac{n}{\ln(n) \sqrt{\dd}}$.  And partition $B$ into $m \leq n/ \dd$ sets $B_1, B_2, \ldots, B_{m}$ each of size $\dd$.  For each $x \in A$ and each $1 \leq j \leq m$, we fully correlate all the edges from $x$ to $B_j$, which is to say $x$ is adjacent to the entirety of $B_j$ with probability $p$ (and otherwise, there are no edges between $x$ and $B_j$).

Now we also partition $A$ into $t = |A| / \sqrt{\dd} = \dfrac{n}{\ln(n)\dd}$ sets $A_1, A_2, \ldots , A_{t}$ each of size $\sqrt{\dd}$.  For each $1 \leq i \leq j \leq t$, fully correlate all the edges between $A_i$ and $A_j$ (in the case $i=j$, this means each $A_i$ is either a complete graph or an empty graph).

Let $S \subseteq A$ denote the set of vertices in $A$ that have no edges to any other vertices of $A$.

\textbf{Claim:} With high probability $|S| \geq |A| \varepsilon/2$.

\textit{Proof of this claim:} For each $1 \leq i \leq j \leq t$, say the ordered pair $(i,j)$ is \textit{on} iff our random graph has edges between $A_i$ and $A_j$.  Each pair is `on' with probability precisely $p$, and these events are independent.  Moreover, there are precisely ${t \choose 2} + t = {t+1 \choose 2}$ pairs $(i,j)$ where $1 \leq i \leq j \leq t$.  Thus, the total number of `on' pairs is a binomial random variable with mean $p{t+1 \choose 2} < (1-\varepsilon)(t+1)/2$ and variance $p(1-p) {t+1 \choose 2}$.  So the chance that the number of `on' pairs is greater than $t(1-\varepsilon   /2)/2$ is at most (by Chebyshev's inequality) $C \varepsilon / (p {t+1 \choose 2})$.  And by Markov's inequality, this probability is also at most $C \varepsilon p {t+1 \choose 2} / t$.  The geometric mean of these upper bounds is $C \varepsilon / \sqrt{t}$, and since $t \to \infty$, at least one of these upper bounds must tend to 0.  So, almost surely, there are at most $t(1-\varepsilon/2)/2$ `on' pairs, meaning there are at least $t\varepsilon/2$ choices of $k$ where the set $A_k$ has no edges inside $A$.  Thus, $|S| \geq (t \varepsilon/2) \sqrt{\dd} = |A|\varepsilon/2$, proving the claim.

Now for each $x \in S$, we would have that $x$ is isolated in the graph iff there are no edges from $x$ to $B$.  Then \textit{after conditioning on $S$}, each vertex of $S$ is isolated with probability exactly $(1-p)^{m}$, and these events are independent.  Because $p \to 0$, we know $(1-p) = \exp(-p(1+o(1)))$.  Thus, if $I$ denotes the total number of isolated vertices in $S$, then $I$ is a binomial random variable with mean
\begin{eqnarray*}
\E[I] &=& |S|(1-p)^m = \exp[\ln|S| -(1+o(1))pm] \geq \exp[\ln(|A|\varepsilon/2) -(1+o(1))pn/\dd]\\
&\geq& \exp[\ln(n\varepsilon/(2\ln(n) \sqrt{\dd})) -(1+o(1))(1-\varepsilon)\ln(n/\sqrt{\dd})]\\
&\geq& \dfrac{\varepsilon(n/\sqrt{\dd})^{\varepsilon -o(1)}}{2 \ln(n)} \to \infty.
\end{eqnarray*}
Thus, by Chebyshev's inequality, we have
\[
\PP(I = 0) \leq \dfrac{\text{Var}(I)}{(\E[I])^2} = \dfrac{(1- (1-p)^m)}{\E[I]} \leq \dfrac{1}{\E[I]} \to 0.
\]
So this graph almost surely has an isolated vertex, meaning it is not connected.
\end{proof}

\section{Proofs of degree concentration and jumbledness}\label{section:degrees and jumbled}

\subsection*{Proof of Theorem \ref{theorem:degree concentration}: degree concentration}
\reminder{\degreeConcentrationTheorem*}
\begin{proof}
Suppose $t = 4 \sqrt{\dd np \ln(n)}$.  For each vertex $u$, consider $A = \{u\}$ and $B = V(G) \setminus \{u\}$.  The degree of $u$ is thus equal to $e(A,B)$, and we have
\[
\PP \Big( |e(A,B) - np| \geq t \Big) \leq \PP \Big( |e(A,B) - (n-1)p| \geq t-p \Big).
\]
By our choice of $t$---and the assumption that $p > \dd \ln(n)/n$---we see that $t > 4 \dd \ln(n) \geq 4 \ln(n)$, which means that for all sufficiently large values of $n$, we have $t-p > t-1 > 3.9 \sqrt{\dd np \ln(n)}$.  Moreover, we have $t-p < t \leq 4np$.  Applying Janson's inequality---our Lemma \ref{lemma:Janson inequalities}, \eqref{Janson:Bernstein}---yields
\begin{eqnarray*}
\PP \Big( |e(A,B) - np| \geq t \Big) &\leq& \PP \Big( |e(A,B) - (n-1)p| \geq 3.9 \sqrt{\dd np \ln(n)} \Big)\\
&\leq& 2 \exp \left( \dfrac{-(8/50) (3.9)^2 \dd np \ln(n)}{\dd ((n-1)p + t/3)}\right) \leq 2 \exp \left( \dfrac{-(8/50) (3.9)^2 \ln(n)}{1 + 4/3}\right)\\
&\leq& 2 n^{-1.04}.
\end{eqnarray*}
Summing over the $n$ choices for the vertex $u$, we see that the expected number of vertices whose degree is not in the desired range is at most $2 n^{-0.04}$, which tends to $0$.  Thus, the probability that there is such a vertex also tends to $0$.
\end{proof}

\subsection*{Proof of Theorem \ref{theorem:jumbledness}: jumbledness}\label{subsection:jumbledness proof}
\reminder{\jumblednessTheorem*}
Our proof is similar to a typical proof for the Erd\H{o}s--R\'enyi graph (e.g., Lemma 8 of \cite{haxell}), with the main exception being that we rely on Janson's inequality in lieu of a standard Chernoff bound.
\begin{proof}
For each pair of sets $(A,B)$, let us first establish some notation used throughout.  Say $|A| =a$, $|B| =b$, and define $\mu := abp$ and $t := C \sqrt{abnp\dd}$, where $C\geq 10$.

Say a pair of sets $(A,B)$ is \textit{bad} iff $\Big | e(A,B) - abp \Big| \geq t$.  We need only argue that with high probability, there are no bad sets satisfying $1 \leq a \leq b \leq n$.  For ease of presentation, we define the following families of sets:
\begin{itemize}
    \item $\mathcal{F} = \{(A,B) \ : \ 1 \leq a \leq b \ \ \ A,B \subseteq V\}$,
    \item $\mathcal{F}_0 = \left \{(A,B) \in \mathcal{F} \ : \ t > \dfrac{C^2}{14} anp \right \}$,
    \item $\mathcal{F}_1 = \left \{(A,B) \in \mathcal{F} \ : \ t < 12 \mu \dfrac{n/b}{\ln(en/b)} \right\}$,
    \item $\mathcal{F}_2 = \mathcal{F} \setminus (\mathcal{F}_0 \cup \mathcal{F}_1) = \left\{(A,B) \in \mathcal{F} \ : \ 12 \mu \dfrac{n/b}{\ln(en/b)} \leq t \leq  anp \dfrac{C^2}{14}\right\}$.
\end{itemize}
We argue that the expected number of bad sets in each of these families is $o(1).$

\paragraph*{Case 0, $\mathcal{F}_0$:} First suppose $(A,B) \in \mathcal{F}_0$.  This means we may assume $t > \dfrac{C^2}{14} anp$.  But by assumption, we know that with high probability, every vertex of $G$ has degree at most $(C^2 / 14) np$.  Therefore, $e(A,B) \leq a (C^2/14) np < t$.  And since $abp \leq anp \leq (C^2 /14) anp < t$, this means $|e(A,B) - abp| \leq \max\{e(A,B), abp\} < t.$  So because of the maximum degree condition, there are no bad sets in $\mathcal{F}_0$.

\paragraph*{Case 1, $\mathcal{F}_1$:} Now suppose $(A,B) \in \mathcal{F}_1$.  This means we may assume $t \leq 12 \mu \dfrac{n/b}{\ln(en/b)}$.  And we can say that $\mu + t/3 \leq \left(1 + \dfrac{12}{3} \cdot \dfrac{n/b}{\ln(en/b)} \right)\mu \leq \left(5 \cdot \dfrac{n/b}{\ln(en/b)} \right)\mu$.  So by using \eqref{Janson:Bernstein}, we have
\[
\pr(\text{$(A,B)$ is bad}) \leq 2 \exp \left[ \dfrac{-8 t^2}{50 \dd (\mu + t/3) }\right] \leq 2 \exp \left[ \dfrac{-8 t^2 b\ln(en/b)}{250 \dd n \mu}\right].
\]
For each $1 \leq b$, there are at most $\sum_{a=1}^{b} {n \choose a} {n \choose b} \leq \sum_{a=1} ^{b} (en/a)^a (en/b)^b \leq b(en/b)^{2b} \leq (en/b)^{3b}$ pairs of sets $(A,B) \in \mathcal{F}$ with $|B| = b$.  Thus, for all \((A,B)\in\mathcal{F}_1\), the expected number of bad sets is at most
\begin{eqnarray*}
2 \sum_{b=1}^{n} \exp[3b \ln(en/b)] \exp \left[ \dfrac{-8 t^2 b\ln(en/b)}{250 \dd n \mu}\right]
&=& 2 \sum_{b=1}^{n} \exp \left[b \ln(en/b) \left(3 - \dfrac{8 C^2 }{250} \right) \right] \\ &\leq& 2 \sum_{b=1}^{n} \exp \left[-b \ln(en/b)/5\right],
\end{eqnarray*}
where the last inequality is justified by the fact that $C \geq 10$.  Since the value of this summation tends to $0$, this shows that the expected number of bad sets in $\mathcal{F}_1$ goes to 0.

\paragraph*{Case 2, $\mathcal{F}_2$:} For all $(A,B) \in \mathcal{F}_2$ we know $t/\mu \geq 12 \dfrac{n/b}{\ln(en/b)}$, meaning $\ln(t/\mu) \geq \ln(ne/b)+ \ln(12/e) - \ln(\ln(en/b))$.  Moreover, for all $1 \leq x$ we have $\ln(12/e)-\ln(\ln(ex))>-\ln(ex)/10$, which means $\ln(t/\mu) \geq 0.9 \ln(ne/b)$.  Since for all $y > 0$, we have $\phi(y) = (1+y)\ln(1+y)-y \geq (5y/8) \ln(5y/4)$, this yields

\[
\phi \left( \dfrac{4t}{5 \mu}\right) \geq \dfrac{1}{2} \dfrac{t}{\mu} \ln(t/\mu) \geq \dfrac{t}{\mu} \cdot \dfrac{(0.9)\ln(en/b)}{2} = (0.45) \dfrac{t \ln(en/b)}{\mu}.
\]

Therefore, for all $(A,B) \in \mathcal{F}_2$, by applying \eqref{Janson:with phi} we have
\begin{eqnarray*}
\PP (\text{$(A,B)$ is bad}) &\leq& 2\exp \left[ \dfrac{-\mu}{2 \dd } \varphi\left(\dfrac{4t}{5\mu} \right) \right] \leq 2\exp \left[ \dfrac{-(0.225)\mu}{\dd } \dfrac{t}{\mu} \ln(en/b) \right]\\
&=& 2\exp \left[ -(0.225)C b \ln(en/b) \sqrt{\dfrac{apn}{\dd b}} \right] \leq 2\exp \left[ -3.15 \cdot b \ln(en/b)\right],
\end{eqnarray*}
where the last inequality follows from the fact that $t \leq anp \dfrac{C^2}{14}$, which means $\dfrac{14}{C} \leq \sqrt{\dfrac{anp}{\dd b}}$.

As before, for each value $b$, there are at most $\exp[3b \ln(ne/b)]$ pairs $(A,B) \in \mathcal{F}_2$.  And thus, the expected number of bad pairs $(A,B) \in \mathcal{F}_2$ is at most
\[
2 \sum_{b=1}^{n} \exp[3b \ln(en/b)] \exp \left[ -3.15 \cdot b \ln(en/b)\right] = 2 \sum_{b=1}^{n} \exp[-(0.15)b \ln(en/b)] = o(1).
\]

Therefore, since the expected number of bad pairs in $\mathcal{F}_0, \mathcal{F}_1,$ and $\mathcal{F}_2$ each tends to $0$, this means that the probability that there exists a bad pair also tends to $0$.
\end{proof}

\section{Proof of Theorem \ref{theorem:connectivity}: connectivity}\label{section:connectivity proof}
\reminder{\connectivityTheorem*}

\begin{proof}
First we note that claim (ii) of Theorem \ref{theorem:connectivity} is rather immediate.  Namely, a connected graph must have at least $n-1$ edges, but by Markov's inequality, $\PP [|E(G)| \geq n-1] \leq \dfrac{\E[|E(G)|]}{n-1} = np/2,$ which tends to 0 by assumption.

We now turn our attention to claim (i), whose proof is almost identical to that of the standard Erd\H{o}s--R\'enyi graph.  For this, call a set $A \subseteq V$ \textit{bad} iff $1 \leq |A| \leq n/2$ and $e(A, V \setminus A) = 0$.  In order for $e(A, V \setminus A)$ to be 0, all of the $|A| (n-|A|)$ edges leaving $A$ must be absent.  But among these, there must be at least $|A| (n-|A|) / \dd$ which are independent.  Therefore we have
\begin{eqnarray*}
\E[\text{\# bad sets}] &=& \sum_{A \subseteq V} \PP(A \text{ is bad}) = \sum_{a=1} ^{n/2} {n \choose a} (1-p)^{a(n-a)/\dd}\\
&\leq& \sum_{a=1} ^{n/2} \exp \left[a \ln(en/a) -p \dfrac{a(n-a)}{\dd} \right] \leq \sum_{a=1} ^{n/2} \left(\exp \left[C\left(2-f(n) \right) \right] \right) ^a,
\end{eqnarray*}
which tends to $0$ since $f(n) \to \infty$.  Thus there are almost surely no bad sets, and $G$ is connected.
\end{proof}

\section{Proof of Theorem \ref{theorem:subgraph counts}: subgraph containment}\label{section:subgraph counts}
\reminder{\subgraphTheorem*}
\begin{proof}
Following aspects of the approach of \cite{mario_paper}, we restrict our attention to the presence of subgraphs whose edge sets are suitably uncorrelated.  For this, we need to define an appropriate notion allowing us to deal with the possible dependence of edges.  Since $G$ is a $d$-dependent random graph, this means that for each edge $e$, there is some set $S(e)$ of size at most $d$ such that the presence of $e$ is independent of the set of edges outside of $S(e) \cup \{e\}$.  With this, we say $e_1$ relies on $e_2$ iff $e_1 \in S(e_2)$ or $e_1 = e_2$.  We write this as $\rel{e_1}{e_2}$

Let $\mathcal{H}$ denote the set of all copies of $H$ appearing in the complete graph on $n$ vertices.  For each $S \in \mathcal{H}$, we say $S$ is \textit{internally uncorrelated} iff none of the edges within $S$ rely on any of the others.  Let $\mathcal{H}' \subseteq \mathcal{H}$ denote the set of internally uncorrelated graphs $S$ (i.e., the internally uncorrelated copies of $H$ in the complete graph on $n$ vertices).

If we select an element $S \in \mathcal{H}$ uniformly at random, then the expected number of pairs of edges $(e_1,e_2)$ in $S$ where $e_1 \neq e_2$ and $\rel{e_1}{e_2}$ is at most
\[
\sum_{e_1} \sum_{ \rel{e_1}{e_2} } \PP \Big(\{e_1,e_2\} \subseteq E(S) \Big) \leq d |E(H)| \cdot \max_{e,f} \PP (f \in E(S) | e \in E(S)) \leq d |E(H)| \dfrac{|V(H)|}{n},
\]
where the last inequality is because the event $f \in E(S)$ conditioned on $e \in E(S)$ involves the inclusion of at least one more vertex in $S$, and inclusion of vertices in $S$ is negatively correlated.

Therefore, we have that the number of internally uncorrelated elements of $\mathcal{H}$ is at least
\[
|\mathcal{H}'| \geq |\mathcal{H}| \left( 1 - \dfrac{d |E(H)| |V(H)|}{n} \right) \geq |\mathcal{H}|/10,
\]
which is also essentially Lemma 11 of \cite{mario_paper}.

For each $S \in \mathcal{H}'$, let $X_S$ denote the indicator random variable for the event that all the edges of $S$ are present in the $d$-dependent random graph $G$, and let $X = \sum_{S \in \mathcal{H}'} X_S$. 
 By construction of $\mathcal{H}'$, for each $S \in \mathcal{H}'$, we have $\E[X_S] = p^{|E(H)|}$, so $\E[X] = |\mathcal{H}'| p^{|E(H)|}$.

Deviating from the approach in \cite{mario_paper}, we now directly turn our attention to estimating the mean and variance of $X$.  Suppose $S \in \mathcal{H}'$ is fixed.  For each subset of edges $\Gamma \subseteq S$, let $\mathcal{B}_{\Gamma}$ denote the collection of graphs $T \in \mathcal{H}'$ such that the edges of $T$ can be partitioned as $T_{0} \cup T_{1}$ where (1) $T_0$ is isomorphic to $\Gamma$, (2) each edge of $T_0$ relies on at least one edge of $S$, and (3) the edges of $T_{1}$ do not rely on any of the edges of $S$.

For each $\Gamma \subseteq S$, we can bound $|\mathcal{B}_{\Gamma}|$ as follows.  Let $f(\Gamma)$ denote the minimum number of edges in a subgraph of $\Gamma$ spanning the same vertex set (i.e., the \textit{edge cover number} of $\Gamma$).  Then $T_0$ can be chosen in at most $(2(d+1))^{f(\Gamma)}$ ways---because each edge in the minimal edge cover must be selected from a set of the form $\{e\} \cup M(e)$, and we can decide which vertex corresponds which end point in at most 2 ways.  Having done so, we could then specify $T_1$ by selecting the set of the remaining vertices in $V(T) \setminus V(T_0)$ and then selecting the edges in one of at most $|V(H)|! / |\text{Aut}(H)|$ ways.  Thus we have
\begin{eqnarray*}
|\mathcal{B}_{\Gamma}| &\leq& (2(d+1)) ^{f(\Gamma)} \cdot {n \choose {|V(H)| - |V(\Gamma)|}} \cdot \dfrac{|V(H)|!}{|\text{Aut}(H)|}\\
&=& (2(d+1)) ^{f(\Gamma)} \cdot \dfrac{{n \choose {|V(H)| - |V(\Gamma)|}}}{{n \choose |V(H)|}} \cdot |\mathcal{H}| \leq (2(d+1)) ^{f(\Gamma)} n^{-|V(\Gamma)|} \cdot \left(\dfrac{|V(H)|}{1-|V(H)|/n} \right)^{|V(\Gamma)|} \cdot |\mathcal{H}|\\
&\leq& (2(d+1)) ^{f(\Gamma)} n^{-|V(\Gamma)|} \cdot \left(10|V(H)|\right)^{|V(\Gamma)|} \cdot |\mathcal{H}|,
\end{eqnarray*}
where the last line is justified by the assumption $|V(H)|/n < 0.9$.

Finally, suppose $T \in \mathcal{B}_{\Gamma}$.  If $E(\Gamma) = \emptyset$, then $\text{Cov}(X_S, X_T) = 0$ since the edges in $T$ would not rely on any edges of $S$.  Otherwise, we have
\begin{eqnarray*}
\text{Cov}(X_S, X_T) &\leq& \E[X_S X_T] \leq p^{|E(S)|} \E[X_T | X_S = 1]\\
&\leq& p^{|E(S)|} p^{|E(T)| - |E(T_0)|} = p^{2|E(H)| - |E(\Gamma)|}.
\end{eqnarray*}
Thus, for $X = \sum_{S \in \mathcal{H}'} X_S$, we have
\begin{eqnarray*}
\text{Var}(X) &=&  \sum_{S, T \in \mathcal{H}'} \text{Cov}(X_S, X_T) \leq \sum_{S \in \mathcal{H}'} \sum_{\emptyset \neq \Gamma \subseteq E(H)} |\mathcal{B}_{\Gamma}| \cdot p^{2|E(H)| - |E(\Gamma)|}\\
&\leq& \sum_{S \in \mathcal{H}'} \sum_{\emptyset \neq \Gamma \subseteq E(H)} (2(d+1)) ^{f(\Gamma)} n^{-|V(\Gamma)|} \cdot \left(10|V(H)|\right)^{|V(\Gamma)|} \cdot |\mathcal{H}| \cdot p^{2|E(H)| - |E(\Gamma)|}\\
&=& \left(|\mathcal{H}'| p ^{|E(H)|} \right) ^2 \dfrac{|\mathcal{H}|}{|\mathcal{H}'|} \sum_{\emptyset \neq \Gamma \subseteq E(H)} \left( \dfrac{10 |V(H)|}{n}\right)^{|V(\Gamma)|} \dfrac{(2(d+1)) ^{f(\Gamma)}}{p^{|E(\Gamma)|}}.
\end{eqnarray*}
Since $\E[X] = |\mathcal{H}'| p^{|E(H)|}$ and $|\mathcal{H}'| \geq (0.1) |\mathcal{H}|$, we have by Chebyshev's inequality
\[
\PP(X=0) \leq \dfrac{\text{Var}(X)}{(\E [X])^2} \leq 10 \sum_{\emptyset \neq \Gamma \subseteq E(H)} \left( \dfrac{20 |V(H)|}{n}\right)^{|V(\Gamma)|} \dfrac{(d+1)^{f(\Gamma)}}{p^{|E(\Gamma)|}}. \qedhere
\]
\end{proof}

\section{Conclusion and Open Questions}
In this paper we study random graph properties when edges have limited local dependence.  In particular, we examine jumbledness, connectivity, and subgraph containment for $d$-dependent random graphs.  Many interesting open problems remain.  We highlight a few below.

As discussed in the introduction, tight threshold functions cannot exist as in the \er case \cite{thresholds}.  Instead, we believe the following more nuanced phenomenon should hold.
\begin{conj}[Existence of threshold functions]
  For every monotone graph property $\cP$ and all $d = O(n/\ln(n))$, there exist functions $L(n,d)$ and $U(n,d)$ such that the following hold:
  \begin{enumerate}
      \item[(i)] If $p \ll L(n,d)$ then for all $d$-dependent random graph distributions $G \sim \gnpd$, $G$ almost surely does not have property $\cP$.
      \item[(ii)] If $p \gg U(n,d)$ then for all $d$-dependent random graph distributions $G \sim \gnpd$, $G$ almost surely has property $\cP$.
      \item[(iii)] If $L(n,d) \ll p \ll U(n,d)$, then there exists a $d$-dependent random graph distribution $H_1$ and a $d$-dependent random graph distribution $H_2$ such that $H_1$ almost surely has property $\cP$ and $H_2$ almost surely does not have property $\cP$.
  \end{enumerate}
    
\end{conj}

Moreover, we expect the difference between the lower and upper thresholds will be bounded by a function depending only on $d$.
\begin{conj}
    For all monotone graph properties $\cP$, there exists a function $f(d)$ depending only on $d$ such that $U(n,d)/L(n,d) = O(f(d))$.
\end{conj}

One way to interpret our results on connectivity is that introducing local edge dependencies allows us to correlate edges in such a way that a random graph remains unconnected even with higher $p$, up to $p = O(d\ln(n)/n)$.  We don't know if it is possible to \emph{anticorrelate} edges such that a graph almost surely is connected for smaller $p$.  We conjecture this is not possible:

\begin{conj}\label{conj:lower-c-thr}
    For all $d = O(n/\ln(n))$ and all $p \ll \ln(n)/n$, and for any $d$-dependent random graph distribution $G \sim \gnpd$, $G$ is almost surely not connected.
\end{conj}
Couched in the language of threshold functions described above, Conjecture~\ref{conj:lower-c-thr} states that for the connectivity property, we have $L(n,d) = \Omega(\ln(n)/n)$.  This is the highest possible value of $L(n,d)$ since the \er random graph distribution is almost surely connected when $p \gg \ln(n)/n$.

Brody and Sanchez~\cite{mario_paper} showed a similar phenomenon holds for clique number---it's possible to \emph{correlate} edges to drive up the clique number of a dependent random graph, but it is not possible to \emph{anti-correlate} edges to make the clique number much smaller than it would be for the \er random graph. One possible reason for this phenomenon is because having a large clique is in a sense a local property of a graph, whereas \emph{not} having a large clique is a more global property.  We conjecture that this phenomenon is quite general---for global graph properties, $d$-dependent random graphs behave similarly to \er random graph distributions, whereas for local graph properties, introducing local dependencies can change the behavior of the graph.

It would be interesting to find similar threshold functions for subgraph containment.
\begin{prob}
For each fixed $H$, find a lower threshold function $L_{H}(n, d)$ with the property that for any $G \sim \gnpd$, if $p \ll L(n,d)$, then almost surely $G$ will not contain any copies of $H$.
\end{prob}
\noindent Addressing this even in the special case that $H$ is a triangle would likely be quite interesting.

Finally, for large $d$, we remark that we naturally lose all control over what's going on.  For instance, say there are integers $a,m \leq d$ where $pm = a$ and $m$ divides ${n \choose 2}$.  Then we can find a $d$-dependent random graph where the number of edges is always equal to exactly $p {n \choose 2}$.  Namely, partition the ${n \choose 2}$ edges into blocks of size $m$ and uniformly select $a$ edges from each block.  We suspect that in order to be able to make meaningful statements about $d$-dependent random graphs, we need $pd \leq 1$ and $d\leq n/\log(n)$, but we are unsure how to frame this as a precise conjecture.

\bibliographystyle{siam}
\bibliography{refs.bib}
\end{document}